\documentclass[fleqn,12pt]{article}
\usepackage{amsmath,amssymb,amsthm,xcolor,framed,esint,dsfont}
\usepackage[english]{babel}
\usepackage[margin=3cm]{geometry}
\usepackage{csquotes}
\usepackage{hyperref}
\usepackage{enumitem}  
\usepackage{tikz}
\usetikzlibrary{calc}

\numberwithin{equation}{section}

\def\Xint#1{\mathchoice
	{\XXint\displaystyle\textstyle{#1}}%
	{\XXint\textstyle\scriptstyle{#1}}%
	{\XXint\scriptstyle\scriptscriptstyle{#1}}%
	{\XXint\scriptscriptstyle\scriptscriptstyle{#1}}%
	\!\int}

\def\XXint#1#2#3{{\setbox0=\hbox{$#1{#2#3}{\int}$}
		\vcenter{\hbox{$#2#3$}}\kern-.5\wd0}}

\newcommand{\R}{{\mathbb R}}

\newcommand{\ep}{\varepsilon}
\newcommand{\lt}{\left}
\newcommand{\rt}{\right}
\newcommand{\na}{\nabla}
\newcommand{\nn}{\nonumber}

\newcommand{\e}{\varepsilon}

\newcommand{\ti}{\tilde}

\newcommand{\BI}{\mathcal{B}}

\DeclareMathOperator{\dv}{div}
\DeclareMathOperator{\supp}{supp}
\DeclareMathOperator{\dist}{dist}
\DeclareMathOperator{\card}{card}

\newtheorem{thm}{Theorem}
\numberwithin{thm}{section}
\newtheorem{prop}[thm]{Proposition}
\newtheorem{lem}[thm]{Lemma}

\theoremstyle{definition}
\newtheorem{rem}[thm]{Remark}

\title{Another regularizing property of the 2D eikonal equation}
\date{}
\author{Xavier Lamy\footnote{Institut de Math\'ematiques de Toulouse, UMR 5219, Universit\'e de Toulouse, CNRS, UPS
		IMT, F-31062 Toulouse Cedex 9, France. Email: Xavier.Lamy@math.univ-toulouse.fr}
	\and Andrew Lorent\footnote{Department of Mathematical Sciences, University of Cincinnati, Cincinnati, OH 45221, USA. Email: lorentaw@uc.edu} 
	\and Guanying Peng\footnote{Department of Mathematical Sciences, Worcester Polytechnic Institute, Worcester, MA 01609, USA. Email: gpeng@wpi.edu}}

\begin{document}

\maketitle

\begin{abstract}
A weak solution of the two-dimensional eikonal equation amounts to a vector field $m\colon\Omega\subset\mathbb R^2\to\mathbb R^2$ such that $|m|=1$ a.e. and $\mathrm{div}\,m=0$ in $\mathcal D'(\Omega)$.
It is known that, if $m$  has some low regularity, e.g., continuous or $W^{1/3,3}$, then $m$ is automatically more regular: locally Lipschitz outside a locally finite set. 
A long-standing conjecture by Aviles and Giga, 
if true, would imply the same regularizing effect  under the Besov regularity assumption $m\in B^{1/3}_{p,\infty}$
for $p>3$.
In this note we establish that regularizing effect in the borderline case $p=6$, above which the Besov regularity assumption implies continuity. 
If the domain is a disk and $m$ satisfies tangent boundary conditions,
we also prove this for $p$ slightly below $6$.
\end{abstract}

\section{Introduction}

Let $\Omega\subset\R^2$ an open set and $m\colon \Omega\to\R^2$
a weak solution of the eikonal equation
\begin{align}\label{eq:eik}
|m|=1\quad\text{a.e. in }\Omega,
\qquad 
\dv m=0\quad\text{ in }\mathcal D'(\Omega)\,.
\end{align}
In a simply connected domain, this is  equivalent to the usual eikonal equation $|\nabla u|=1$ for $u\colon\Omega\to\R$ such
that $im=\nabla u$, where $i$ denotes rotation by $\pi/2$.

We are interested in  regularizing features of the eikonal equation \eqref{eq:eik},
of the form: 
if a weak solution $m$ has a given low regularity, then $m$ is locally Lipschitz outside a locally finite set. 
The latter property
corresponds to being a 
 \emph{zero-energy state} of the Aviles-Giga energy, as defined and characterized in \cite{JOP02}.
We are aware of two instances  of this regularizing effect:
 \begin{itemize}
 \item
 If $m$ is continuous, then $m$ is locally Lipschitz. This follows e.g. from \cite{dafermos06} or \cite[Lemma~2.2]{CC10}, see also \cite{ignat24short}.
 \item
 If $m$ is $W^{1/3,3}$, then $m$ is a zero-energy state \cite{DLI15}. See also \cite{ignat12,ignat12survey} for the $W^{1/2,2}$ case.
 \end{itemize}
The results in \cite{LP23facto} make it natural to conjecture 
another instance of this regularizing effect:
\begin{align}\label{eq:conj}
\left.
\begin{aligned}
m\text{ solves \eqref{eq:eik} } 
\\
m\in B^{1/3}_{p,\infty}(\Omega)\text{ for some }p>3
\end{aligned}
\right\rbrace
\quad\Rightarrow
\quad
m\text{ is a zero-energy state}\,,
\end{align}
where the Besov regularity $m\in B^{1/3}_{p,\infty}$ is defined by
\begin{align*}
m\in B^{1/3}_{p,\infty}(\Omega)
\quad\Leftrightarrow
\quad
\sup_{|h|>0} \frac{1}{|h|^{1/3}}\|m(\cdot +h)-m\|_{L^p(\Omega\cap(\Omega-h))} <\infty\,.
\end{align*} 
More precisely, the validity of the conjecture \eqref{eq:conj} is a necessary condition for the validity of the Aviles-Giga conjecture \cite[p.9]{AG87}, see the discussion in \cite[\S~1.2]{LP23facto}.
The range $p>3$ in \eqref{eq:conj} is sharp, since pure jump solutions $m=m^+\mathbf 1_{x_1>0} +m^-\mathbf 1_{x_1<0}$ 
(with $m^\pm \in\mathbb S^1$ s.t. 
$m^+_1 =m^-_1$)
 belong to the space $B^{1/3}_{3,\infty}$,
which plays a critical role in the Aviles-Giga conjecture \cite{GL20}.
 

For $p>6$,  $B^{1/3}_{p,\infty}$ regularity implies continuity \cite[\S~2.7.1]{triebel83}, 
so  \eqref{eq:conj} is true.
For $3<p\leq 6$,
this does not follow directly from the 
already known regularizing effects,
since
$B^{1/3}_{p,\infty}$ regularity does not imply continuity,
nor $W^{1/3,3}$ regularity.\footnote{
An example showing that $B^{1/3}_{p,\infty}\not\subset W^{1/3,3}= B^{1/3}_{3,3}$ in any bounded domain can be constructed using a wavelet basis, as recalled e.g. in \cite[Corollary~2.17(ii)]{CLY24}.}
In this note we present a short argument solving the borderline case $p=6$.

\begin{thm}
\label{T1}
Let $\Omega\subset \mathbb{R}^2$ be an open set,
and $m\colon\Omega\to\R^2$ a weak solution of the eikonal equation \eqref{eq:eik}.
If $m\in B^{1/3}_{6,\infty}(\Omega)$,
then $m$ is locally Lipschitz outside a locally finite set.
\end{thm}

We rely on the characterization of zero-energy solutions of \eqref{eq:eik}
established in \cite{JOP02}: a weak solution $m$ is locally Lipschitz outside a locally finite set if
\begin{align}\label{eq:zero_ent}
&
\dv\Phi(m)=0\quad\text{in }\mathcal D'(\Omega)\,,\qquad\forall \Phi\in\mathrm{ENT}\,,
\\
&
\mathrm{ENT}=\left\lbrace
\Phi\in C^2(\mathbb S^1;\R^2)\colon \frac{d}{d\theta}\Phi(e^{i\theta}) \cdot e^{i\theta}=0\quad\forall\theta\in\R
\right\rbrace\,.
\nonumber
\end{align}
Maps in $\mathrm{ENT}$ are  called entropies, 
and for a weak solution $m$ the distributions $\dv\Phi(m)$ are the corresponding entropy productions.
Entropies are characterized by the fact that all smooth solutions of the eikonal equation have zero entropy production \eqref{eq:zero_ent}.
They were introduced in \cite{DKMO01} to study compactness properties of sequences with bounded Aviles-Giga energy.

\begin{rem}
In \cite[Theorem~1.4]{LP23facto} it is shown that, for a
unit divergence-free vector field $m$ as in \eqref{eq:eik}
and $3<p\leq 4$,
local  $B^{1/3}_{p,\infty}$ regularity is equivalent to local $L^{\frac p 3}$ integrability of all entropy productions $\dv\Phi(m)$.
That equivalence could not be true for $p>6$,
since in that case $B^{1/3}_{p,\infty}$ maps are continuous, 
but the discontinuous map  $m(x)=ix/|x|$ has zero entropy productions.
There is however no direct obstruction to it being true for $p=6$.
In that case, Theorem~\ref{T1} would imply that any solution of \eqref{eq:eik} with $L^2$ entropy productions is a zero-energy state, thus partly answering \cite[Conjecture~1.5]{LP23facto}.
\end{rem}

We let $m_\e$ denote the convolution 
 $m_\e =m*\rho_\e$ where $\rho_\e(z)=\e^{-2}\rho(z/\e)$  and 
 $\rho$ is a fixed kernel supported in $B_1$
 with
 $\int_{B_1}\rho = 1$, $0\leq\rho\leq 1 $ and $|\nabla\rho|\leq 1$.
The main ingredient in our proof of Theorem~\ref{T1} 
shows Lipschitz regularity under the assumptions that $m_\e$ stays away from zero and that entropy productions are in $L^p$ for some $p>1$.
\begin{prop}\label{p:rigid_B3pconvol}
Let $m\in B^{1/3}_{3,\infty}(B_2;\mathbb S^1)$ with $\dv m=0$ and $p>1$.
Assume that $\dv\Phi(m)\in L^p(B_2)$ for all $\Phi\in\mathrm{ENT}$,
and
\begin{align*}
\limsup_{\e\to 0} \left(\inf_{B_1} |m_\e| \right) >0.
\end{align*}
Then $m$ is Lipschitz in $B_{1/2}$.
\end{prop}
The proof of Theorem~\ref{T1} follows from Proposition~\ref{p:rigid_B3pconvol}
combined with
the following property of $B^{1/3}_{6,\infty}$.
\begin{lem}\label{l:badconvolB6}
For any $m\in B^{1/3}_{6,\infty}(\Omega;\mathbb S^1)$, the set of points $x\in \Omega$
such that 
\begin{align*}
\limsup_{\e\to 0}\left( \inf_{B_r(x)}|m_\e|\right) =0\qquad\forall \,r \in \lt(0,\mathrm{dist}\lt(x,\Omega^c\rt) \rt),
\end{align*}
is Lebesgue-negligible.
\end{lem}

\begin{proof}[Proof of Theorem~\ref{T1} from Proposition~\ref{p:rigid_B3pconvol} and Lemma~\ref{l:badconvolB6}]
Denote by $X\subset \Omega$ the negligible set of points in Lemma~\ref{l:badconvolB6}.
If $x\in \Omega\setminus X$, there exists $r>0$ such that $B_{2r}(x)\subset \Omega$ and
\begin{align*}
\limsup_{\e\to 0} \lt(\inf_{B_r(x)}|m_\e|\rt)  >0,
\end{align*}
thus we can apply Proposition~\ref{p:rigid_B3pconvol} to $m$ appropriately rescaled in $B_{2r}(x)$, and we deduce that $m$ is Lipschitz in $B_{r/2}(x)$. 
This implies $\dv\Phi(m)=0$ in $B_{r/2}(x)$, for any entropy $\Phi$.
Since this is valid for a.e. $x\in \Omega$ and $\dv\Phi(m)\in L^2$ thanks to \cite[Proposition~4.1]{LP23facto}, we infer that $\dv\Phi(m)=0$ in $\Omega$,
and may invoke \cite[Theorem~1.3]{JOP02} to conclude.
\end{proof}

\subsection{Improved estimate under tangent boundary conditions in a disk}

If $\Omega$ is simply connected and $m$ solves \eqref{eq:eik}, then there exists $u\colon\Omega\to\R$ such that $im=\nabla u$ and $|\nabla u|=1$ a.e. in $\Omega$.
Motivated by physical considerations, and assuming $\Omega$ has Lipschitz boundary, natural boundary conditions for this function $u$ are
\begin{align*}
	u=0 \quad\text{and}\quad
	\frac{\partial u}{\partial n} =-1\quad\text{ on }\partial \Omega,
\end{align*}
where $\partial u/\partial n$ denotes the exterior normal derivative \cite{JK00}.
In terms of $m$ this corresponds to the tangential boundary condition
\begin{align}\label{eq:tangent_bc}
m=\tau_{\partial\Omega}\quad\text{ on }\partial\Omega\,,
\end{align}
where $\tau_{\partial\Omega}=i n_{\partial\Omega}$ is the counterclockwise unit tangent to $\Omega$.

For a solution $m\in B^{1/3}_{3,\infty}(\Omega)$, 
the kinetic formulation \cite{JP01,GL20} allows to define a one-sided trace 
of $m$ on $\partial\Omega$  by the arguments in \cite{vasseur01} or \cite{DLO03}, and therefore make sense of this
tangential boundary condition.
If $m\in B^{1/3}_{p,\infty}(\Omega)$ for some $p>3$, then $m$ automatically has a trace on $\partial\Omega$ (see e.g. \cite[\S~3.3.3]{triebel83}).
Here we will replace these  trace considerations by requiring that $m$ is extended equal to  $i\nabla \dist(\cdot,\partial\Omega)$ outside $\Omega$.

Specializing to the case of the disk $\Omega=B_1$, 
we therefore consider $m\colon B_4\to\mathbb S^1$
 such that $\dv m=0$ and
\begin{align}\label{eq:cond_ext_B1_m}
	m(x)=i\frac{x}{|x|}\qquad\forall x\in B_4\setminus B_1\,.
\end{align}
%
Under these boundary conditions, we have
\begin{thm}\label{t:bc}
	Let $m\in B^{1/3}_{q,\infty}(B_4;\mathbb S^1)$ for some $(47+\sqrt{553})/12 < q \leq 6$ such that $\dv m=0$
	and \eqref{eq:cond_ext_B1_m} holds. Then
	\begin{align}\label{eq:int_B1_m}
		m(x)=i\frac{x}{|x|}\qquad\forall x\in B_1\setminus \{0\}\,.
	\end{align}
\end{thm}
The  vortex configuration given in \eqref{eq:int_B1_m} is the only zero-energy state under the boundary condition \eqref{eq:cond_ext_B1_m} over a disk, as characterized in \cite[Theorem 1.2]{JOP02}.

\begin{rem}\label{r:tangentbc}
The ideas we use to prove Theorem~\ref{t:bc} 
can be elaborated on to obtain  that, 
under general assumptions on a smooth domain $\Omega$,
 if $m$ solves \eqref{eq:eik} in $\Omega$ with tangential boundary conditions \eqref{eq:tangent_bc} and 
has the regularity $m\in B^{1/3}_{q,\infty}$ for some $q > (47+\sqrt{553})/12$, then $\Omega$ must be a disk and $m$ given by \eqref{eq:int_B1_m}.
This works for instance if $\Omega$ is uniformly convex or analytic.
We present only Theorem~\ref{t:bc} in order to keep the presentation short and not-too-technical.
\end{rem}

\paragraph{Plan of the article.}
In \S~\ref{s:proof_rigidB3pconvol} we present the proof of Proposition~\ref{p:rigid_B3pconvol},
relying on a lemma 
proved in  \S~\ref{s:proof_characepsBr}.
In \S~\ref{s:proof_badconvolB6} we give the proof of Lemma~\ref{l:badconvolB6}, and in \S~\ref{s:refined_B1/3} we prove a refined version of a regularity estimate from \cite{GL20}. 
Finally, the proof of Theorem~\ref{t:bc} is given in \S~\ref{s:tbc}.

\paragraph{Notation.}
The notation $A\lesssim B$ stands for the existence of an absolute constant $C>0$ such that $A\leq C\, B$.

\paragraph{Acknowledgments.}
 XL received support from ANR project ANR-22-CE40-0006. 
AL  was supported in part by NSF grant DMS-2406283. 
GP was supported in part by NSF grant DMS-2206291.

\section{Proof of Proposition~\ref{p:rigid_B3pconvol}}\label{s:proof_rigidB3pconvol}

The proof of Proposition~\ref{p:rigid_B3pconvol} 
uses the link between entropy productions and a kinetic formulation discovered in \cite{JP01}, and further explored in \cite{GL20,LP23facto}.
Relevant to us are the following properties.

\begin{prop}[{\cite{GL20},\cite[Proposition~4.2]{LP23facto}}]
\label{p:kin}
Let $m
\in B^{1/3}_{3,\infty}(\Omega;\mathbb \R^2)$ a weak solution of the eikonal equation \eqref{eq:eik}.
There exists $\sigma\in \mathcal M(\Omega\times\mathbb S^1)$ such that
\begin{align*}
e^{is}\cdot\nabla_x\mathbf 1_{m(x)\cdot e^{is}>0} =\partial_s\sigma\quad\text{in }\mathcal D'(\Omega\times\mathbb S^1).
\end{align*}
If  $\dv\Phi(m)\in L^p(\Omega)$ for all $\Phi\in \mathrm{ENT}$ and some $p>1$, 
then $\sigma\in L^p(\Omega;\mathcal M(\mathbb S^1))$, that is,
 the measure 
 \begin{align}\label{eq:nu}
 \nu=(\mathrm{proj}_\Omega )_\sharp |\sigma| = \int_{\mathbb S^1} |\sigma|(\cdot, ds)\in\mathcal M(\Omega)\,,
 \end{align}
has an $L^p$ density   with respect to the Lebesgue measure.
\end{prop}

With these notations, 
the main ingredient in the proof of Proposition~\ref{p:rigid_B3pconvol} is the following lemma, 
which shows that any integral curve of the curl-free vector field $im_\e$  must be almost straight, 
provided $\nu\in L^p$ and $|m_\e|$ stays away from zero along the curve.

\begin{lem}\label{l:epsrigid_convol}
Let $m\in B^{1/3}_{3,\infty}( B_2;\mathbb S^1)$ with 
$\dv m=0$, hence $im=\nabla u$ for some 
\mbox{1-Lipschitz} function $u\colon B_2\to\R$.
Assume that $\nu\in L^p(B_2)$ for
some
 $p>1$, where
$\nu$ is defined in \eqref{eq:nu}. 
Let $\e\in (0,1)$.
If, for some $T>0$ and $c_0\in (0,1)$, 
there is an integral curve
$\gamma\colon [0,T]\to B_1$ 
such that
\begin{align*}
&
\dot\gamma =\nabla u_\e(\gamma)\quad\text{in }[0,T]\,,
\\
\text{and}
\quad
&
|m_\e|\geq c_0>0\quad\text{ on }\gamma([0,T])\,,
\end{align*}
then we have
\begin{align}
&
u_\e(\gamma(T))-u_\e(\gamma(0))
\geq
 |\gamma(T)-\gamma(0)|
- \delta
\,,
\label{eq:epsrigid_convol1}
\\
\text{and}
\quad
&
\gamma([0,T])\subset [\gamma(0),\gamma(T)] +B_{\delta + \sqrt{\delta T}}\,,
\label{eq:epsrigid_convol2}
\\
\text{where}\quad
&
\delta
=C
(\|\nu\|_{L^p}/c_0^2)^{\frac{p}{9p-6}}\e^{\frac{p-1}{9p-6}}T^{\frac{9p-7}{9p-6} }\,,
\nonumber
\end{align}
for some absolute constant $C>0$.
\end{lem}

\begin{proof}[Proof of Proposition~\ref{p:rigid_B3pconvol} from Lemma~\ref{l:epsrigid_convol}]
We let $u\colon B_2\to \R$ be such that $\nabla u=im$.
By assumption, there exist $c_0>0$ and a sequence $\e\to 0$ such that $|m_\e|\geq c_0$ in $B_1$.
For any $\e$ in that sequence, consider the maximal integral curve $\gamma_\e\colon (S_\e,T_\e)\to B_1$ solving
\begin{align*}
\gamma_\e(0)=0,\quad \dot\gamma_\e=\nabla u_\e(\gamma_\e).
\end{align*}
Since $u_\e$ is $1$-Lipschitz and $(d/dt)[u_\e(\gamma_\e)]=|\nabla u_\e|^2(\gamma_\e)\geq c_0^2$, 
we have
$T_\e-S_\e\leq 2/c_0^2$.
We fix 
$S_\e^*<0<T^*_\e$ such that $\gamma_\e((S^*_\e,T_\e^*))\subset B_{1/2}$ and $X_\e=\gamma_\e(S^*_\e)\in \partial B_{1/2}$, $Y_\e=\gamma_\e(T_\e^*)\in \partial B_{1/2}$.
Thanks to Lemma~\ref{l:epsrigid_convol} applied on the time intervals $[S^*_\e,0]$ and $[0,T^*_\e]$, these points $X_\e,Y_\e\in\partial B_{1/2}$ satisfy
\begin{align*}
u_\e(Y_\e)-u_\e(X_\e)\geq 1 - c \e^{\frac{p-1}{9p-6}},
\end{align*}
for some constant $c$ depending on $\|\nu\|_{L^p}$ and $c_0$.
Extracting a subsequence $\e\to 0$, we deduce the existence of $X,Y\in\partial B_{1/2}$ such that
$u(Y)-u(X)\geq 1$. Since $u$ is 1-Lipschitz and $|X-Y|\leq 1$, this implies that
$|X-Y|=1$ and $u$ is affine with slope 1 along the segment $[X,Y]$.
Rescaling, we can apply this argument to deduce that, 
for any $x\in B_1$ such that $B_{2r}(x)\subset B_1$, there exists a direction $w_x\in\mathbb S^1$ such that $u$ restricted to $x+[-rw_x,rw_x]$ is affine with slope 1.
This implies that $\nabla u$ is constant (equal to $w_x$) along that segment.
Two such segments
starting from points in $B_{1/2}$ cannot cross inside $B_{2/3}$, and this implies that m is locally
Lipschitz,
see e.g.  the proof of \cite[Lemma~5.1]{JOP02}.
\end{proof}

\begin{proof}[Proof of Lemma~\ref{l:epsrigid_convol}]
Since $u_\e$ is $1$-Lipschitz and $(d/dt)[u_\e(\gamma)]=|\nabla u_\e|^2(\gamma)\geq c_0^2$, 
we know that 
\begin{align}\label{eq:gammac0}
|\gamma(t)-\gamma(s)|\geq c_0^2|t-s|
\quad\text{ for all }s,t\in [0,T]\,.
\end{align}
For $r\in [\e,\min(1/4,T)]$, to be fixed later, 
we 
decompose the time interval $[0,T]$ into $N-1$ subintervals $[t_j,t_{j+1}]$, with
\begin{align*}
0=t_1 < t_2 <\cdots < t_N =T,\quad \frac r2 \leq t_{j+1}-t_j\leq r,\quad N\leq \frac{2T}{r}\,.
\end{align*} 
Setting $X_j=\gamma(t_j)$, 
we have
$\gamma([t_j,t_{j+1}])\subset B_r(X_j)$ since $|\dot\gamma|\leq 1$.
Moreover, the inequality \eqref{eq:gammac0}
implies $|X_i-X_j|\geq c_0^2|i-j|r/2$, and
 ensures therefore the bounded intersection property
\begin{align}\label{eq:coverB3r}
\sum_{j=1}^N \mathbf 1_{B_{4r}(X_j)} \lesssim \frac{1}{c_0^2} \mathbf 1_{\bigcup_{j=1}^N B_{4r}(X_j)}
\lesssim\frac{1}{c_0^2}\mathbf 1_{\gamma([0,T])+B_{4r}}\,.
\end{align}
Applying the estimate \eqref{eq:ue} in
Lemma~\ref{l:characepsBr} in the next section,
on each time interval $[t_j,t_{j+1}]$, 
we find, for any $\alpha>0$,
\begin{align*}
u_\e(X_{j+1})-u_\e(X_j)
&
\geq
(1-\sqrt{\alpha})|X_{j+1}-X_j| - \frac C\alpha (\nu(B_{4r}(X_j))+\e^{1/2}r^{1/2})\,,
\end{align*}
where $C>0$ is a generic absolute constant which may change from line to line in what follows.
Summing over $j$, using that $N\leq 2T/r$,
that $|\gamma(T)-\gamma(0)|\leq T$ and the property \eqref{eq:coverB3r},
 we deduce
\begin{align*}
&u_\e(\gamma(T))-u_\e(\gamma(0))
\\
&
\geq (1-\sqrt{\alpha})|\gamma(T)-\gamma(0)|
-\frac C\alpha\sum_{j=1}^N \nu(B_{4r}(X_j)) -C\frac T\alpha
\sqrt{\frac \e r}
\\
&\geq
|\gamma(T)-\gamma(0)|
-\sqrt\alpha T
-\frac C{\alpha}
\bigg(\frac{
\nu (
\gamma([0,T])+B_{4r}
)}{c_0^2} + 
 T \sqrt{\frac\e r}\,
 \bigg)
\,.
\end{align*}
Choosing
\begin{align*}
\alpha =
\bigg(\frac{
\nu (
\gamma([0,T])+B_{4r}
)}{c_0^2T} + 
\sqrt{\frac\e r}\,
 \bigg)^{\frac 23}\,,
\end{align*}
we obtain
\begin{align*}
u_\e(\gamma(T))-u_\e(\gamma(0))
&\geq
 |\gamma(T)-\gamma(0)|
-
CT \bigg(\frac{
\nu (
\gamma([0,T])+B_{4r}
)}{c_0^2T} + 
\sqrt{\frac\e r}\,
 \bigg)^{\frac 13}\,.
\end{align*}
Using that $\nu$ has an $L^p$ density and the Lebesgue measure of $\gamma([0,T])+B_{4r}$ is at most $16rT$, this implies
\begin{align*}
u_\e(\gamma(T))-u_\e(\gamma(0))
&
\geq |\gamma(T)-\gamma(0)|
-
CT \bigg(
\frac{\|\nu\|_{L^p(B_2)}}{c_0^2}
\frac{
(rT)^{1-\frac 1p}}{T} + 
\sqrt{\frac\e r}\,
 \bigg)^{\frac 13}\,.
\end{align*}
Finally we choose
\begin{align*}
r=\bigg(\frac{c_0^2}{\|\nu\|_{L^p(B_2)}}\bigg)^{\frac{2p}{3p-2}}\e^{\frac{p}{3p-2}} T^{\frac 2{3p-2}},
\end{align*}
which gives
\begin{align*}
u_\e(\gamma(T))-u_\e(\gamma(0))
&
\geq
 |\gamma(T)-\gamma(0)|
-
C
\bigg(\frac{\|\nu\|_{L^p(B_2)}}{c_0^2}\bigg)^{\frac{p}{9p-6}}
\e^{\frac{p-1}{9p-6}}T^{\frac{9p-7}{9p-6} }\,.
\end{align*}
This proves \eqref{eq:epsrigid_convol1}. 
To show \eqref{eq:epsrigid_convol2}, for any $t\in (0,T)$ we apply \eqref{eq:epsrigid_convol1} on the intervals $[0,t]$ and $[t,T]$ and,
since $u_\e$ is 1-Lipschitz, deduce the chain of inequalities
\begin{align*}
|\gamma(T)-\gamma(0)|
&
\leq |\gamma(T)-\gamma(t)| +|\gamma(t)-\gamma(0)|
\\
&
\leq u_\e(\gamma(T))-u_\e(\gamma(t)) +u_\e(\gamma(t))-u_\e(\gamma(0)) + 2\delta
\\
&
=u_\e(\gamma(T))-u_\e(\gamma(0)) +2\delta
\\
&
\leq |\gamma(T)-\gamma(0)|+2\delta\,.
\end{align*}
So we have the approximate reverse triangle inequality
\begin{align*}
|\gamma(T)-\gamma(t)| +|\gamma(t)-\gamma(0)|
\leq |\gamma(T)-\gamma(0)|+2\delta\,,
\end{align*}
which implies that $\gamma(t)$ must be close to the segment $[\gamma(0),\gamma(T)]$.
More precisely, 
let $d=\dist(\gamma(t),[\gamma(0),\gamma(T)])=|\gamma(t)-X|$ for some $X\in [\gamma(0),\gamma(T)]$.
Assume first that $X\notin \lbrace\gamma(0),\gamma(T)\rbrace$.
Then $\ell_1=|\gamma(0)-X|$ and $\ell_2=|\gamma(T)-X|$ satisfy
\begin{align*}
&|\gamma(t)-\gamma(0)| = \sqrt{\ell_1^2 +d^2}\,,
\quad
|\gamma(T)-\gamma(t)| = \sqrt{\ell_2^2 +d^2}\,,
\\
&
\text{and } \ell_1+\ell_2 = |\gamma(T)-\gamma(0)|\,.
\end{align*}
If $d\leq |\gamma(T)-\gamma(0)|$,  
then, 
using that 
$\sqrt{1+x}\geq 1+x/3$ for all $x\in [0,1]$, we deduce
\begin{align*}
|\gamma(T)-\gamma(t)| +|\gamma(t)-\gamma(0)|
&
\geq
\ell_1\sqrt{1+d^2/\ell_1^2} +\ell_2\sqrt{1+d^2/\ell_2^2}
\\
&
\geq (\ell_1+\ell_2)
\sqrt{
1+\frac{d^2}{(\ell_1+\ell_2)^2}}
\\
&\geq 
|\gamma(T)-\gamma(0)|
+
\frac{d^2}{3|\gamma(T)-\gamma(0)|}\,,
\end{align*}
hence $d\leq \sqrt{6\delta T}$.
If $d\geq |\gamma(T)-\gamma(0)|$, then we have
\begin{align*}
|\gamma(T)-\gamma(t)| +|\gamma(t)-\gamma(0)|\geq 2d \geq |\gamma(T)-\gamma(0)|+d,
\end{align*}
hence $d\leq 2\delta$. 
And if $X\in\lbrace \gamma(0),\gamma(T)\rbrace$, then we also have
\begin{align*}
|\gamma(T)-\gamma(t)|+|\gamma(t)-\gamma(0)|\geq  |\gamma(T)-\gamma(0)|+d\,,
\end{align*}
and $d\leq 2\delta$.
In all cases we have $d\leq 2\delta +\sqrt{6\delta T}$,
and, after adjusting the absolute constant $C$, this gives \eqref{eq:epsrigid_convol2}.
\end{proof}

\section{Proof of Lemma~\ref{l:characepsBr}}\label{s:proof_characepsBr}

In this section we prove the following.

\begin{lem}\label{l:characepsBr}
Let $r\in (0,1)$
and
 $m\in B^{1/3}_{3,\infty} (B_{4r};\mathbb S^1)$ such that $\dv m=0$,
 hence $im=\nabla u$ for some $1$-Lipschitz function $u\colon B_{4r}\to\R$.
 
For $0<\e\leq r$, let $\gamma\colon [t_1,t_2]\to B_r$ solve $\dot\gamma =\nabla u_\e(\gamma)$,
and denote $X_j=\gamma(t_j)$, $j=1,2$.
Then
 for all $\alpha>0$ we have
 \begin{align}\label{eq:ue}
u_\e(X_2)-u_\e(X_1)
&
 \geq 
(1-\sqrt{\alpha}) |X_2-X_1|
-\frac C\alpha (\nu(B_{4r})+\e^{1/2}r^{1/2})
\nn\\
&
\quad
+\int_{\gamma([t_1,t_2])\cap \lbrace |m_\e|\leq 1-\sqrt\alpha\rbrace}
|m_\e| \, d\mathcal H^1\,,
 \end{align}
where $C>0$ is an absolute constant.
\end{lem}

The proof of Lemma~\ref{l:characepsBr} relies on the two next lemmas, where we denote by $D^h$ the finite difference operator
\begin{align}\label{eq:Dh}
D^h f(x)=f(x+h)-f(x)\,,
\end{align}
for $h\in\R^2$.

\begin{lem}\label{l:mod1unifproj}
	For any $m\colon B_{4r}\to\mathbb S^1$ and $0<\e\leq r$ we have
	\begin{align*}
		\int_{-r}^r \left( \sup_{\lbrace x_1\rbrace\times (-r,r)} (1-|m_\e|)^2 \right)\, dx_1
		\lesssim \sup_{|h|\leq\e}\frac{1}{|h|}\int_{B_{2r}}|D^hm|^3\, dx.
	\end{align*}
\end{lem}

\begin{lem}\label{l:besovnuQr}
	For any $r\in (0,1)$ and $m\in B^{1/3}_{3,\infty} (B_{4r};\mathbb S^1)$ such that $\dv m=0$, we have
	\begin{align*}
		\frac{1}{|h|}\int_{B_{2r}}|D^h m|^3\, dx \lesssim \nu(B_{4r})+r^{1/2}|h|^{1/2}
		\quad
		\forall h\in B_r,
	\end{align*}
	where $\nu$ is defined in \eqref{eq:nu}.
\end{lem}

As a consequence of Lemmas~\ref{l:mod1unifproj} and \ref{l:besovnuQr},
under the assumptions of Lemma~\ref{l:characepsBr} we 
have
\begin{align}\label{eq:mod1unifprojnuQr}
\int_{-r}^r\left( \sup_{\lbrace x_1\rbrace\times (-r,r)} 
(1-|m_\e|)^2 \right)\, dx_1
\lesssim
 \nu(B_{4r})+\e^{1/2} r^{1/2}\,,
\end{align}
and we can proceed to prove Lemma~\ref{l:characepsBr}.

\begin{proof}[Proof of Lemma~\ref{l:characepsBr}]
We write
\begin{align*}
u_\e(X_2)-u_\e(X_1)
&
=\int_{t_1}^{t_2} \dot\gamma(t)\cdot \nabla u_\e(\gamma(t))\, dt
\\
&
=\int_{\gamma([t_1,t_2])}|m_\e|\, d\mathcal H^1
\\
&
\geq (1-\sqrt\alpha)\mathcal H^1\left(
\gamma([t_1,t_2])\cap\lbrace (1-|m_\e|)^2\leq\alpha\rbrace\right)
\\
&\quad
+
\int_{\gamma([t_1,t_2])\cap \lbrace |m_\e|\leq 1-\sqrt\alpha\rbrace}
|m_\e| \, d\mathcal H^1\,.
\end{align*}
Then we assume without loss of generality that $X_2-X_1$ is along the $x_1$-axis,
denote by $\pi_1$ the projection onto it,
and use
\eqref{eq:mod1unifprojnuQr}
to estimate
\begin{align*}
&
\mathcal H^1\left(
\gamma([t_1,t_2])\cap\lbrace (1-|m_\e|)^2\leq\alpha\rbrace\right)
\\
&
\geq \mathcal H^1\left(
\pi_1
\left[
\gamma([t_1,t_2])\cap\lbrace (1-|m_\e|)^2\leq\alpha\rbrace\right]
\right)
\\
&
\geq |X_2-X_1|-\frac{C}{\alpha}\int_{-r}^r
\left( \sup_{\lbrace x_1\rbrace\times (-r,r)} 
(1-|m_\e|)^2 \right)\, dx_1
\\
&
\geq |X_2-X_1|
-\frac{C}{\alpha}\left(\nu(B_{4r})+\e^{1/2} r^{1/2}\right)\,.
\end{align*}
Plugging this into the above estimate concludes the proof.
\end{proof}

Finally we give the proofs of Lemmas~\ref{l:mod1unifproj} and \ref{l:besovnuQr}.

\begin{proof}[Proof of Lemma~\ref{l:mod1unifproj}]
For any fixed $x_1\in (-r,r)$ we have
\begin{align*}
\sup_{\lbrace x_1\rbrace\times (-r,r)}
(1-|m_\e|)^2
&
 \leq \frac 1r \int_{-r}^r (1-|m_\e|)^2\, dx_2
+\int_{-r}^r 
\left|\frac{d}{dx_2}\left[(1-|m_\e|)^2\right]\right|\, 
dx_2
\\
&
\leq \frac 1r \int_{-r}^r 
(1-|m_\e|)^{3/2}\, 
dx_2 
+2 
\int_{-r}^r (1-|m_\e|)|\nabla m_\e|\, dx_2.
\end{align*}
Integrating with respect to $x_1$
 we deduce
\begin{align*}
&\int_{-r}^r 
\left( \sup_{\lbrace x_1\rbrace\times (-r,r)} (1-|m_\e|)^2 \right)\, dx_1
\\
&
\leq
\frac 1r \int_{B_{2r}}(1-|m_\e|)^{3/2}\, dx 
 +2\int_{B_{2r}}(1-|m_\e|)|\nabla m_\e|\, dx
 \\
 &
\leq
\frac 1r \int_{B_{2r}}(1-|m_\e|)^{3/2}\, dx 
 +2\left(\int_{B_{2r}}(1-|m_\e|)^{3/2}\, dx \right)^{\frac 23}
 \left(\int_{B_{2r}}|\nabla m_\e|^3\, dx\right)^{\frac 13}.
\end{align*}
The conclusion follows from the estimates
\begin{align*}
	\int_{B_{2r}}(1-|m_\e|)^{3/2}\, dx 
	&
	\lesssim 
	\e \sup_{|h|\leq\e}\frac{1}{|h|} \int_{B_{2r}}|D^h m|^3\, dx\,,
	\\
	\int_{B_{2r}}|\nabla m_\e|^3\, dx
	&
	\lesssim \frac{1}{\e^2}\sup_{|h|\leq\e}\frac{1}{|h|} \int_{B_{2r}}|D^h m|^3\, dx\,,
\end{align*}
see e.g.  \cite[Step~6 in Proposition~3]{DLI15},
and the fact that $\e\leq r$.
\end{proof}

\begin{proof}[Proof of Lemma~\ref{l:besovnuQr}]
This follows from 
keeping track more precisely of each step in the proof of \cite[Proposition~3.7]{GL20}, see Lemma~\ref{l:refined_B1/3}.
Choosing a test function $\phi$ in Lemma~\ref{l:refined_B1/3} such that $\mathbf 1_{B_{2r}}\leq 
\phi\leq
\mathbf 1_{B_{3r}}$ and 
$|\nabla\phi|\lesssim 1/r$ gives Lemma~\ref{l:besovnuQr}.
\end{proof}

\section{Proof of Lemma~\ref{l:badconvolB6}}\label{s:proof_badconvolB6}

%
%

Lemma~\ref{l:badconvolB6} follows from a classical covering argument which provides the following.

\begin{lem} 
\label{L10}
Let  $m\in B^{s}_{q,\infty}(\Omega;\mathbb S^1)$ for some $s\in (0,1)$ and $q\geq 1$  and $U\subset \subset \Omega$. 
For any $0<\e < \dist(U,\Omega^c)/3$, there is a finite set $X^{U}_{\e}\subset U$ 
such that
\begin{align*}
 \lt|m_{\e}\rt|\geq \frac{1}{2}\quad\text{ in } U
 \setminus \bigcup_{x\in X_\e^U} B_{5\e}(x)\,, 
 \quad
 \text{ and }
 \quad
\card(X_\e^U)\lesssim \|m\|_{B^{s}_{q,\infty}(\Omega)}^q \e^{sq-2}\,.
\end{align*}
\end{lem}

\begin{proof}[Proof of Lemma~\ref{l:badconvolB6} from Lemma~\ref{L10}]
Thanks to Lemma~\ref{L10} 
applied to $m\in B^{\frac 13}_{6,\infty}$ we can select a sequence $\e_k\to 0$ and a finite set $X_*^U$ such that 
\begin{align*}
\dist(X_{\e_k}^U,X_*^U)\to 0\qquad\text{as } k\to\infty.
\end{align*}
For $x\in U\setminus X_*^U$ and   $r>0$ such that $B_{2r}(x)\subset U\setminus X_*^U$, 
we have
\begin{align*}
B_r(x)\subset U\setminus \bigcup_{x\in X_\e^U} B_{5\e_k}(x)\,,
\end{align*}
for large enough $k$, 
and therefore 
\begin{align*}
\limsup_{\e\to 0} \left(\inf_{B_r(x)} |m_\e|\right)\geq \frac 12.
\end{align*}
We conclude that the set of points considered in Lemma~\ref{l:badconvolB6} is
 locally finite.
\end{proof}

\begin{proof}[Proof of Lemma~\ref{L10}] 
Given $\alpha>0$, to be fixed later, define 
\begin{align}
\label{eqzx61}
\BI_{\e}^U:=\lt\{x\in U: \Xint{-}_{B_{\e}}\Xint{-}_{B_{\e}}\lt|m(x+y)-m(x+z)\rt|^{q} dy dz> \alpha\rt\}.
\end{align}
Since $|m|=1$ a.e. in $\Omega$ and $|\rho_\e|\leq 1/\e^2$, 
for any $x\in U\setminus \BI_\e^U$ we have
\begin{align*}
\big|1-|m_\e(x)|\big|
&
=\Big|\int_{B_\e}  \left(  |m(x-z)|\, -\Big|\int_{B_\e}m(x-y)\, \rho_\e(y) dy\Big| \right) \rho_\e(z)\, dz  \Big|
\\
&
\leq 
\int_{B_\e}\int_{B_\e}|m(x-z)-m(x-y)|\,  \rho_\e(y)\rho_\e(z) dy dz
\\
&
\lesssim 
\left(\Xint{-}_{B_\e}\Xint{-}_{B_\e}|m(x-z)-m(x-y)|^{q}\, dy dz
\right)^{\frac{1}{ q }}
\lesssim
\alpha^{1/q}\,.
\end{align*}
The last inequality follows from the fact that $x\in U\setminus \BI_\e^U$ and the definition \eqref{eqzx61} of $\BI_\e^U$.
Hence, we may fix a small enough absolute constant $\alpha >0$ 
so that
\begin{align}\label{eq:mepsBeps}
|m_\e|\geq \frac 12 \qquad\text{ in }U\setminus \BI_\e^U\,.
\end{align}
For any $x\in \BI_\e^U$ and  $\ti{x}\in B_{\e}(x)$ we have 
\begin{align*}
&\Xint{-}_{B_{2\e}}\Xint{-}_{B_{2\e}}\lt|m\lt(\ti{x}+y\rt)-m\lt(\ti{x}+z\rt)\rt|^{q} dy dz\nn\\
&\geq \frac{1}{16\pi^2\e^4}\int_{B_{\e}}\int_{B_{\e}}\lt|m\lt(x+y\rt)-m\lt(x+z\rt)\rt|^{q} dy dz\geq \frac{\alpha}{16}\,,
\end{align*}
so that
\begin{align}
\label{eqzx61.7}
\int_{B_{\e}(x)}\Xint{-}_{B_{2\e}}\Xint{-}_{B_{2\e}}\lt|m\lt(\ti{x}+y\rt)-m\lt(\ti{x}+z\rt)\rt|^{q} dy dz \, d \ti{x}\geq \frac{\pi \e^2\alpha}{16}\,,
\qquad\forall x\in\BI_\e^U\,.
\end{align}
By the Vitali covering lemma there exists a finite set
$X_\e^U\subset \BI_{\e}^U$ such that 
\begin{align}
\label{eqzx62}
 \BI_{\e}^U
 \subset \bigcup_{x\in X_\e^U} B_{5 \e}\lt(x\rt)
\end{align}
and the disks $\lt\{ B_{\e}(x)\colon x\in X_\e^U \rt\}$ are pairwise disjoint. 
Recalling \eqref{eqzx61.7} 
  we infer
\begin{align*}
\e^2   \card (X_\e^U) 
&\lesssim
 \int_{U+B_\e} \, \Xint{-}_{B_{2\e}}\Xint{-}_{B_{2\e}}  \lt|m(x+y)-m(x+z)\rt|^{q} 
\, dy dz\, dx
 \lesssim \|m\|_{ B^{s}_{q,\infty}}^{q} \e^{ sq }.
\end{align*}
The last inequality follows from the definition of $ B^{s}_{q,\infty}$ regularity.
This implies the bound
$\card (X_\e^U) \lesssim   \|m\|_{B^{s}_{q,\infty}(\Omega)}^q \e^{sq-2}$.
Combining this with \eqref{eq:mepsBeps} and the inclusion \eqref{eqzx62} concludes the proof.
\end{proof}


%
%
%

%
%

%
%


\section{Refined Besov estimate}\label{s:refined_B1/3}

The proof of $B^{1/3}_{3,\infty}$ regularity in \cite[Proposition~3.7]{GL20}
provides an estimate which can be expressed more precisely than the one stated there.

\begin{lem}\label{l:refined_B1/3}
Let $m
\in B^{1/3}_{3,\infty}(\Omega;\mathbb \R^2)$ a weak solution of the eikonal equation \eqref{eq:eik} 
and $\nu\in\mathcal M(\Omega)$ as in \eqref{eq:nu}.
For any $\phi\in C^1_c(\Omega)$ and $0<\eta<\dist(\supp(\phi),\Omega^c)$, we have
\begin{align}\label{eq:refined_B1/3}
\sup_{|h|\leq \eta}\int_{\Omega}|m(x+h)-m(x)|^3\phi^2(x)\, dx
&
\lesssim
\eta \sup_{|h|\leq \eta}\int_\Omega \phi^2(x+h)\,\nu(dx) 
\nonumber
\\
&\quad
+\eta^{3/2} \int_\Omega |\phi|^{1/2}(x)|\nabla\phi|^{3/2}(x)\, dx\,.
\end{align}
\end{lem}

\begin{proof}[Proof of Lemma~\ref{l:refined_B1/3}]
Recall the finite difference operator $D^h$ defined in \eqref{eq:Dh}. The calculations in \cite[Lemma~3.9]{GL20}
provide an identity for $h$-derivatives  of the quantity
\begin{align*}
\Delta^{\e,\delta}(h,x)
& = \iint_{\mathbb S^1\times\mathbb S^1}
 \varphi_\delta(s-t) D^h\chi_\e(x,t)D^h\chi_\e(x,s)\,e^{it}\wedge e^{is}\, dt\, ds,
\end{align*}
where 
\begin{align*}
&
\varphi_\delta =\varphi *\gamma_\delta,\quad \varphi(t)=\mathbf 1_{\cos(t)\sin(t)>0}-\mathbf 1_{\cos(t)\sin(t)<0},
\\
&
\chi_\e (x,t)=\left(\mathbf 1_{e^{it}\cdot m(x)>0} \right) *\rho_\e(x)\,.
\end{align*}
Here $\gamma_\delta(t)=\delta^{-1}\gamma_{\delta}(t/\delta)$ for a fixed smooth even kernel $\gamma_{\delta}\in C_c^1(-1,1)$, and $\rho_\e$ is a two-dimensional convolution kernel as above.
Thanks to Proposition~\ref{p:kin}, the function $\chi_\e$ solves the kinetic equation
\begin{align*}
e^{it}\cdot\nabla_x \chi_\e =\partial_s\sigma_\e,
\qquad
\sigma_\e =\sigma *_x\rho_\e\,.
\end{align*}
The calculations in  \cite[Lemma~3.9]{GL20} 
imply, for $0<\e<\dist(\supp(\phi),\Omega^c)-\eta$, and
scalar $h\in (-\eta,\eta)$,
\begin{align}\label{eq:dhDelta}
&
\frac{d}{d h}\int_\Omega\Delta^{\e,\delta}(he_1,x)\phi^2(x)\, dx
\nonumber
\\
&
=\int_{\Omega} I^{\e,\delta}(h,x)\phi^2(x)\, dx
 -
 2\int_{\Omega} \phi(x) \nabla\phi(x)\cdot A^{\e,\delta}(h,x)\, dx\,,
\end{align}
where $I^{\e,\delta}$ and $A^{\e,\delta}$ are given by
\begin{align*}
I^{\e,\delta}
&=
-2\int_{\mathbb S^1}
\left( 
\int_{\mathbb S^1} \varphi'_\delta(s-t)\chi_\e(x,s)\sin s \, ds
\right)\, 
\sigma_\e(x+he_1,dt)
\nonumber
\\
&\quad
+
2\int_{\mathbb S^1}
\left(
 \int_{\mathbb S^1}\varphi'_\delta(s-t)\chi_\e(x+he_1,s)\sin s \, ds
 \right)\,
 \sigma_\e(x,dt)\,,
\nonumber
\\
A_1^{\e,\delta}
&=
2\iint_{\mathbb S^1\times\mathbb S^1} \varphi_\delta(s-t)\sin s \, \cos t \,
\chi_\e(x+he_1,s)D^{he_1}\chi_\e(x,t) \, ds dt\,,
\nonumber
\\
A_2^{\e,\delta}
&
=2\iint_{\mathbb S^1\times\mathbb S^1} \varphi_\delta(s-t)\sin s\, \sin t\,
 \chi_\e(x,s)D^{he_1}\chi_\e(x,t)\, dsdt\,.
\nonumber
\end{align*}
In \cite{GL20}, the second component of the vector field $A^{\e,\delta}$ has a slightly different expression but can be put in this form using the fact that $\varphi_\delta$ is an odd function.
Using that $\varphi_\delta'$ is bounded in $L^1(\mathbb S^1)$, that $|\chi_\e|\leq 1$,
and the definition of $\nu =\int_{\mathbb S^1} |\sigma|(\cdot,ds)$, we see that
\begin{align}\label{eq:controlI}
\int_{\Omega} I^{\e,\delta}(h,x)\phi^2(x)\, dx
\lesssim \sup_{|z|<\eta}\int_\Omega \phi^2(x+z)\,\nu(dx)\,.
\end{align}
Further, the vector field $A^{\e,\delta}$
can be rewritten as
\begin{align*}
A^{\e,\delta}(h,x)=\int_{B_\e} \left(F^{\e,\delta,x,h}(m(x+z+he_1))-
F^{\e,\delta,x,h}(m(x+z))\right)\, \rho_\e(z)\, dz,
\end{align*}
for a Lipschitz vector field $F^{\e,\delta,x,h}$ (details can be found in \cite[Lemma~4.10]{LP23facto}), and this implies
\begin{align*}
|A^{\e,\delta}(h,x)|\lesssim \int_{B_\e}|D^{he_1} m(x+z)|\,\rho_\e(z)\, dz\,.
\end{align*}
Using this and \eqref{eq:controlI},
integrating \eqref{eq:dhDelta} with respect to $h$ and passing to the limit as $\e\to 0$ and $\delta\to 0$, we infer
\begin{align*}
\frac 1\eta \int_\Omega\Delta(he_1,x)\phi^2(x)\,dx
&
\lesssim \sup_{|z|\leq\eta}\int_\Omega \phi^2(x+z)\,\nu(dx)
\\
&\quad
+\sup_{|z|\leq\eta}\int_\Omega \phi(x) |D^z m(x)|\, |\nabla\phi|(x)\, dx,
\end{align*}
for all $h\in (-\eta,\eta)$, where $\Delta\gtrsim |D^{he_1} m|^3$
 thanks to \cite[Lemma~3.8]{GL20}.
This estimate does not depend on the specific choice of the direction $e_1$, 
so we deduce
\begin{align*}
\frac 1\eta \sup_{|h|\leq \eta} \int_\Omega |D^h m(x)|^3\phi^2(x)\,dx
&
\lesssim \sup_{|h|\leq \eta}\int_\Omega \phi^2(x+h)\,\nu(dx)
\\
&\quad 
+\sup_{|h|\leq \eta}\int_\Omega \phi(x) |D^h m(x)|\,|\nabla\phi|(x)\, dx.
\end{align*}
Thanks to Young's inequality $3ab\leq a^3 + 2 b^{3/2}$, 
for  any $\lambda>0$ we have 
\begin{align*}
\phi |D^h m|  |\nabla\phi|
\leq \frac 13\,\frac\lambda\eta \phi^2 |D^h m|^3 + 
\frac 23\,\frac{\eta^{1/2}}{\lambda^{1/2}}\phi^{1/2} |\nabla\phi|^{3/2}\,.
\end{align*}
Choosing $\lambda$ small enough allows to absorb the term containing $\phi^2 |D^h m|^3$ into the left-hand side,
and infer \eqref{eq:refined_B1/3}.
\end{proof}

%
%

\section{Proof of Theorem~\ref{t:bc}}\label{s:tbc}

In this section, we give the proof of Theorem~\ref{t:bc}. Recall that $m$ satisfies the boundary condition \eqref{eq:cond_ext_B1_m}, which we copy here:
\begin{align}\label{eq:cond_ext_B1_m'}
	m(x)=i\frac{x}{|x|}\qquad\forall x\in B_4\setminus B_1\,,
\end{align}
Then there is a 1-Lipschitz function $u\colon B_4\to \R$ which satisfies $im=\nabla u$ and
\begin{align}\label{eq:cond_ext_B1_u}
	u(x)=1-|x|\qquad\forall x\in B_4\setminus B_1\,.
\end{align}
In the following lemma, 
we
combine Lemma~\ref{l:epsrigid_convol} with the fact that, in $B_3\setminus B_1$, the vector field $\nabla u$ obtained from \eqref{eq:cond_ext_B1_u}
points towards the origin, 
 to
 show that, if $\nu\in L^p$ for some $p>1$, and $|m_\e|\geq 1/2$ in a
 not-too-thin horizontal  strip  which lies above the origin,
  then we can find many integral curves of $im_\e =\nabla u_\e$ crossing the  strip  from top to bottom.

\begin{lem}\label{l:across_strip}
Let $m\in B^{1/3}_{3,\infty}(B_4;\mathbb S^1)$ such that $\dv m=0$
and \eqref{eq:cond_ext_B1_m'} holds, hence $im=\nabla u$ for some 1-Lipschitz function $u\colon B_4\to\R$ satisfying \eqref{eq:cond_ext_B1_u}.
Assume that $\nu\in L^p(B_2)$ for some $p>1$, where $\nu$ is defined in \eqref{eq:nu}.
Then there exists a constant $K\geq 1$ depending on $\|\nu\|_{L^p}$ and $p$ such that the following holds true.
Let $\e\in (0,1)$ and 
assume that
\begin{align*}
&
|m_\e|\geq \frac 12\quad\text{ in the strip }
S_{a,b}=
\lbrace a<x_2<b\rbrace\cap B_3\,,
\nonumber
\\
&
\text{for some }
0<a<b\leq 1
\text{ such that }
b-a\geq 2K\e^{\alpha}\,,
\nonumber
\\
&\text{where }\alpha =\alpha_p=\frac{p-1}{18p-12}\,.
\end{align*}
Then, for every $\xi\in (-\sqrt{1-b^2},\sqrt{1-b^2})$, there exists
\begin{align*}
x\in B_{K\e^\alpha}((\xi,b))\,,
\end{align*}
and an integral curve $\gamma \colon [0,T]\to B_1$ such that
\begin{align*}
&
\dot\gamma =\nabla u_\e(\gamma)\,,\quad
\gamma(0)=x\,,\quad
\gamma(T)=(\xi',a)\,,
\end{align*}
for some $\xi'\in (-\sqrt{1-a^2},\sqrt{1-a^2})$.
\end{lem}

\begin{proof}[Proof of Lemma~\ref{l:across_strip}]
We assume
\begin{align*}
b-a\geq 2K\e^{\alpha}\,,
\end{align*}
for some constant $K\geq 1$ that will be adjusted during the proof.
We let 
\begin{align*}
y=\frac{a+b}{2}\,,
\end{align*} 
and consider, for any $\zeta\in (-1,1)$, 
the maximal integral curve $\gamma_\zeta\colon (T_1^\zeta,T_2^\zeta)\to S_{a,b}$ 
solving
\begin{align*}
\dot\gamma_\zeta =\nabla u_\e(\gamma_\zeta)\quad
\text{in }(T_1^\zeta,T_2^\zeta)\,,
\qquad \gamma_\zeta(0)=(\zeta,y)\,.
\end{align*}
Since $|m_\e|\geq 1/2$ in $S_{a,b}$, 
we have  $(d/dt)u_\e(\gamma_\zeta) \geq 1/4$ and therefore
\begin{align*}
T_2^\zeta-T_1^\zeta \leq 4 |\gamma_\zeta(T^\zeta_2)-\gamma_\zeta(T^\zeta_1)|\leq 16\,,
\end{align*}
because $u_\e$ is $1$-Lipschitz.
This implies in particular that 
$\gamma_\zeta$ can be extended continuously to $[T_1^\zeta,T_2^\zeta]$ and $\gamma_\zeta(T_i^\zeta)\in\partial S_{a,b}$ for $i=1,2$. 
Since $(\zeta,y)$ lies at distance at least $(b-a)/2$ from $\partial S_{a,b}$ and $\gamma_\zeta$ is $1$-Lipschitz, we deduce also
\begin{align*}
\min(T_2^\zeta,-T_1^\zeta)\geq \frac{b-a}{2}
\geq  K \e^{\alpha}\,.
\end{align*}
%
%
%
Thanks to 
 Lemma~\ref{l:epsrigid_convol},
we know that 
$u_\e$ increases with almost unit speed along 
the curve $\gamma_\zeta$,
which must therefore be almost straight:
\begin{align}
&
u_\e(\gamma_\zeta(t))-u_\e(\gamma_\zeta(s))
\geq
|\gamma_\zeta(t)-\gamma_\zeta(s)|
-\kappa\e^{2\alpha}\,,
\label{eq:gammazeta_uepsincrease}
\\
&
\dist(
\gamma_\zeta([s,t]),[\gamma_\zeta(s),\gamma_\zeta(t)]
)
\leq \frac \kappa 2 \e^{\alpha}
 \qquad\quad\text{for all }s<t\in [T_1^\zeta,T_2^\zeta]\,,
 \label{eq:gammazeta_close_segment}
\end{align}
for some constant $\kappa>0$ depending on $\|\nu\|_{L^p}$ and $p$.
This implies in particular that the image of $\gamma_\zeta$ is contained in a thin band around the line
\begin{align*}
L_\zeta 
&
=(\zeta, y) +\R w_\zeta\,,\quad
w_\zeta 
=\frac{\gamma_\zeta(T_2^\zeta)-\gamma_\zeta(T_1^\zeta)}
{|\gamma_\zeta(T_2^\zeta)-\gamma_\zeta(T_1^\zeta)|}\,,
\end{align*}
namely,
\begin{align}
\gamma_\zeta([T_1^\zeta,T_2^\zeta])
\subset L_\zeta +B_{\kappa\e^\alpha}\,.
\label{eq:gammazeta_thinband}
\end{align}
Next we gather some information about these integral curves. 
First, due to the explicit expression \eqref{eq:cond_ext_B1_u} of $u$  outside $B_1$, 
there
the vector field $\nabla u_\e$ always points towards the inside of $B_1$,
and so any integral curve which intersects $B_1$ must stay in $B_1$ at later times.
 This implies that
\begin{align*}
\gamma_\zeta(T_2^\zeta)\in B_1\cap \partial S_{a,b} =B_1\cap \lbrace x_2 =a\text{ or }b\rbrace\,, 
\qquad
\text{for }|\zeta|\leq \sqrt{1-y^2}\,.
\end{align*}
Second, if $K>2\kappa$, 
then for any $\zeta\in [-\sqrt{1-y^2},\sqrt{1-y^2}]$, the entering point $\gamma_\zeta(T_1^\zeta)$ and the exit point $\gamma_\zeta(T_2^\zeta)$ cannot lie both 
on the top horizontal line $\R\times\lbrace b\rbrace$ 
or both on the bottom horizontal line $\R\times\lbrace a\rbrace$.
Indeed, in that case we would have $L_\zeta =\R\times\lbrace b\rbrace$ or $\R\times\lbrace a\rbrace$,
hence the thin band $L_\zeta +B_{\kappa \e^{ \alpha}}$
would not intersect the horizontal line $\R\times\lbrace y\rbrace$ which contains $\gamma_\zeta(0)$, contradicting the fact that by \eqref{eq:gammazeta_thinband} the image of $\gamma_\zeta$ must be contained in that thin band. 

The explicit expression
 \eqref{eq:cond_ext_B1_u} of $u$ in $B_3\setminus B_1$, also implies that for $(\zeta,y)$ outside $B_1$, the entering point $\gamma_\zeta(T_1^\zeta)$ of $\gamma_\zeta$
  lies on the top horizontal line $\R\times\lbrace b\rbrace$. 
 By the previous remark, for these curves the exit point $\gamma_\zeta(T_2^\zeta)$ must lie on the bottom horizontal line 
 $\R\times\lbrace a\rbrace$.
Since integral curves cannot intersect,
 for each $\zeta\in (-1,1)$ we deduce the alternative
\begin{align}
\label{eq:alternative_gamma_zeta}
&
\gamma_\zeta(T_1^\zeta)\in \R\times\lbrace b\rbrace\text{ and }\gamma_\zeta(T_2^\zeta)\in \R\times\lbrace a\rbrace\,,
\nonumber
\\
\text{or}\quad
&
\gamma_\zeta(T_2^\zeta)\in \R\times\lbrace b\rbrace\text{ and }\gamma_\zeta(T_1^\zeta)\in \R\times\lbrace a\rbrace\,.
\end{align}

\begin{figure}[htbp]
	\centering
	\begin{tikzpicture}
		\usetikzlibrary{calc} 
		
		\draw[thick] (0,1) -- (12.5,1);
		\node at (13.5,1) {\(\mathbb{R} \times \{a\}\)}; 
		
		\draw[thick] (0,5) -- (12.5,5);
		\node at (13.5,5) {\(\mathbb{R} \times \{y\}\)}; 
		
		\draw[thick] (3,5) circle (1); 
		
		\coordinate (center) at (3,5);
		\def\radius{1}
		
		\coordinate (tangent1) at ($(center) + (60:\radius)$); 
		\coordinate (tangent2) at ($(center) + (240:\radius)$); 
		\coordinate (perp) at ($(tangent2) + (0,{-4 + sqrt(3)/2})$); 
		\coordinate (gamma) at (9.3,1);
		
		\draw[thick, blue] ($(tangent1) + ({-0.5*sqrt(3)},0.5)$) -- ($(tangent1) + ({5.1*sqrt(3)},-5.1)$); 
		
		\draw[thick, blue] ($(tangent2) + ({-1.4*sqrt(3)},1.4)$) -- ($(tangent2) + ({3.4*sqrt(3)},-3.4)$); 
		
		\coordinate (B) at ($(perp) + ({sqrt(3)*(4 - sqrt(3)/2)},0)$); 
		\coordinate (A) at ($(center) + (-2,0)$); 
		
		\fill[black] (tangent2) circle (2pt); 
		\fill[black] (center) circle (2pt); 
		\fill[black] (perp) circle (2pt); 
		\fill[black] (B) circle (2pt); 
		\fill[black] (A) circle (2pt);
		\fill[black] (gamma) circle (2pt);
		
		\node[below left] at (tangent2) {\(T\)};
		\node[below left] at (B) {\(B\)}; 
		\node[above] at (center) {\(O\)};
		\node[below left] at (A) {\(A\)};
		\node[below] at (perp) {\(P\)};
		\node[below] at (gamma) {\(\gamma_{\zeta}(T_j^{\zeta})\)};
		\node at (0.5, 5.7) {\(L_\zeta^-\)};
		
		\draw[dashed, thick] (tangent2) -- (perp); 
		\draw[dashed, thick] (center) -- (tangent2); 
		\draw[dashed, thick, red] (center) -- (gamma); 
	\end{tikzpicture}
	\caption{Estimating the horizontal width of the thin band $L_\zeta +B_{\kappa\e^\alpha}$.}
	\label{f1}
\end{figure}
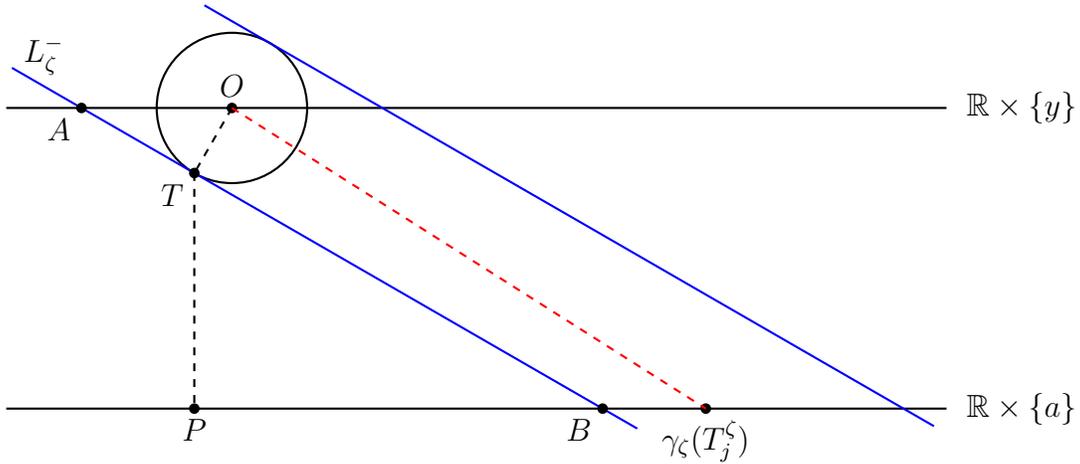

Next we show that, in the alternative \eqref{eq:alternative_gamma_zeta}, the second case actually never happens. 
To this end, given $\zeta\in (-1, 1)$, 
we first estimate the horizontal width
$h_\zeta$ of the thin band $L_\zeta +B_{\kappa\e^\alpha}$,
in terms of the lengths of the segments $[\gamma_\zeta(0),\gamma_{\zeta}(T_j^\zeta)]$, for $j=1,2$.

 Assume without loss of generality that the slope of $L_\zeta$ is negative. Let $A$ and $B$ denote the intersection points of the left-most line $L_\zeta^-$ of the band $L_\zeta +B_{\kappa\e^\alpha}$ with the horizontal lines $\R\times\{y\}$ and $\R\times\{a\}$, respectively (see Figure~\ref{f1}). Further, let $O=\gamma_{\zeta}(0)$, let $T$ be the orthogonal projection of $O$ onto the line $L_\zeta^-$, and let $P$ be the orthogonal projection of $T$ onto the horizontal line $\R\times\{a\}$. Since the triangles $OAT$ and $TBP$ are similar, we have
\begin{equation*}
	\frac{|O-A|}{|O-T|} = \frac{|T-B|}{|T-P|}\,.
\end{equation*}
Denoting by $h_\zeta = 2|O-A|$, and using $|O-T| = \kappa\e^\alpha$ and $|T-P|\geq \frac{b-a}{2} - \kappa\e^\alpha \geq (K-\kappa)\e^\alpha$, we obtain
\begin{equation*}
	h_\zeta \leq \frac{\kappa}{K-\kappa}\,2|T-B| \leq \frac{|T-B|}{2}\,,
\end{equation*}
provided $K>5\kappa$. Now assuming $\gamma_\zeta(T_j^\zeta)\in \R\times\{a\}$, we have $|\gamma_\zeta(0)-\gamma_\zeta(T_j^\zeta)| \geq |T-B|$ and deduce from the above that
\begin{equation}\label{eq:est_h_zeta}
	h_\zeta \leq \frac{|\gamma_\zeta(0)-\gamma_\zeta(T_j^\zeta)|}{2}\,. 
\end{equation}
Repeating the above argument between the horizontal lines $\R\times\{b\}$ and $\R\times\{y\}$, we obtain the estimate \eqref{eq:est_h_zeta} for $j=1, 2$.
\color{black}

For $|\zeta|\geq \sqrt{1-y^2}$ we have seen already that we are in the first case in \eqref{eq:alternative_gamma_zeta}.
Assume by contradiction that the second case does happen.
Then we can find $\zeta<\zeta'\in (-1,1)$ such that $|\zeta'-\zeta|\leq \kappa\e^{ \alpha }$ and
\begin{align*}
\gamma_\zeta(T_1^\zeta),\gamma_{\zeta'}(T_2^{\zeta'})\in \R\times\lbrace b\rbrace\quad\text{ and }\quad\gamma_\zeta(T_2^\zeta),\gamma_{\zeta'}(T_1^{\zeta'})\in \R\times\lbrace a\rbrace\,.
\end{align*}
According to \eqref{eq:gammazeta_thinband} the thin bands
\begin{align*}
L_\zeta +B_{\kappa\e^\alpha}
\quad
\text{and}
\quad 
L_{\zeta'} +B_{\kappa\e^\alpha}\,,
\end{align*}
cannot fully cross inside the horizontal strip $\R\times [a,b]$. 
The size of the horizontal segment formed by the two intersection points of the left-most line of the left thin band and the right-most line of the right thin band
 with the horizontal line $\R\times \lbrace z\rbrace$ is an affine function 
 $h(z)$ of $z$.
So its minimum on $[a,b]$ is attained at $a$ or $b$, and this implies
 \begin{align*}
 	\min\big(h(a), h(b)\big) \leq h(y) \leq h_\zeta + h_{\zeta'}\,, 
 \end{align*}
where the last inequality follows from the assumption $|\zeta'-\zeta|\leq \kappa\e^{ \alpha }$. Assume for instance that the minimum in the above
 left-hand side is attained by $h(a)$, then using the property \eqref{eq:est_h_zeta} we obtain
  \begin{align*}
 	|\gamma_{\zeta}(T_2^\zeta)-\gamma_{\zeta'}(T_1^{\zeta'})|
 \leq h(a) \leq \frac{|\gamma_\zeta(0)-\gamma_\zeta(T_2^\zeta)|}{2} + \frac{|\gamma_{\zeta'}(0)-\gamma_{\zeta'}(T_1^{\zeta'})|}{2}\,. 
 \end{align*}
 Using the increasing property \eqref{eq:gammazeta_uepsincrease} of $u_\e$ along these integral curves, we find
\begin{align*}
&u_\e(\gamma_{\zeta}(T_2^\zeta))-u_\e(\gamma_{\zeta'}(T_1^{\zeta'}))
\\
&
\geq u_\e(\gamma_{\zeta}(T_2^\zeta))-u_\e(\gamma_\zeta(0))
+
u_\e(\gamma_{\zeta'}(0))-u_\e(\gamma_{\zeta'}(T_1^{\zeta'}))
- |\gamma_\zeta(0)-\gamma_{\zeta'}(0)|
\\
&
\geq |\gamma_\zeta(0)-\gamma_\zeta(T_2^\zeta)|+|\gamma_{\zeta'}(0)-\gamma_{\zeta'}(T_1^{\zeta'})|-3\kappa\e^{ \alpha} \,.
\end{align*}
But since $u_\e$ is 1-Lipschitz and $|\gamma_\zeta(0)-\gamma_\zeta(T_2^\zeta)|\geq K\e^{ \alpha}$, $|\gamma_{\zeta'}(0)-\gamma_{\zeta'}(T_1^{\zeta'})|\geq K\e^{ \alpha}$, the above two estimates lead to the contradiction 
\begin{align*}
	K\e^{ \alpha}\leq\frac{|\gamma_\zeta(0)-\gamma_\zeta(T_2^\zeta)|}{2} + \frac{|\gamma_{\zeta'}(0)-\gamma_{\zeta'}(T_1^{\zeta'})|}{2} \leq 3\kappa\e^{ \alpha}\,.
\end{align*}
This demonstrates
 our claim that
all curves $\gamma_\zeta$ are in the first case of the above alternative \eqref{eq:alternative_gamma_zeta}, namely,
\begin{align*}
&
\gamma_\zeta(T_1^\zeta)\in \R\times\lbrace b\rbrace\text{ and }\gamma_\zeta(T_2^\zeta)\in \R\times\lbrace a\rbrace\,,
\qquad\forall \zeta\in (-1,1)\,.
\end{align*}

We obtain therefore two functions $\xi_1,\xi_2\colon (-1,1)\to \R$, characterized by
\begin{align*}
\gamma_\zeta(T_1^\zeta)=(\xi_1(\zeta),b),\qquad
\gamma_\zeta(T_2^\zeta)=(\xi_2(\zeta),a)\,.
\end{align*}
As integral curves cannot cross, both functions $\xi_1,\xi_2$ are monotone increasing, and thanks to the explicit expression \eqref{eq:cond_ext_B1_u} of $u$ outside $B_1$ we 
have
\begin{align*}
&
\xi_1(-\sqrt{1-y^2})<-\sqrt{1-b^2},\quad \xi_1(\sqrt{1-y^2})>\sqrt{1-b^2}\,,
\\
&\xi_2(-\sqrt{1-y^2})>-\sqrt{1-a^2}\,,
\quad \xi_2(\sqrt{1-y^2})<\sqrt{1-a^2}\,.
\end{align*}
Since $\xi_2$ is increasing, this implies $|\xi_2(\zeta)|<\sqrt{1-a^2}$ if $|\zeta|< \sqrt{1-y^2}$.
And since $\xi_1$ is increasing, this implies that
for any $\xi\in (-\sqrt{1-b^2},\sqrt{1-b^2})$,
we can find $\zeta_*\in (-\sqrt{1-y^2},\sqrt{1+y^2})$ such that
\begin{align*}
\lim_{\substack{\zeta\to\zeta_*\\ \zeta <\zeta_*}}\xi_1(\zeta) = \xi_1(\zeta_*^-)\leq \xi \leq \xi_1(\zeta_*^+)
= \lim_{\substack{\zeta\to\zeta_*\\ \zeta >\zeta_*}}\xi_1(\zeta)\,.
\end{align*}
By continuity of the flow generated by $\nabla u_\e$, we can fix
$\delta >0$ and $T_1^*\in (T_1^{\zeta_*},0)$ such that, for all $\zeta\in (\zeta_*-\delta,\zeta_*+\delta)$, we have
\begin{align*}
T_1^\zeta < T_1^*\quad\text{and}
\quad
\gamma_\zeta(T_1^*)\in B_{\kappa\e^\alpha}(\xi_1(\zeta_*)) \,.
\end{align*}
We fix $\zeta' \in (\zeta_*-\delta,\zeta_*)$ and $\zeta''\in (\zeta_*,\zeta_* +\delta)$, so that
 $\xi_1(\zeta')<\xi<\xi_1(\zeta'')$, and
$\gamma_{\zeta'}(T_1^*)$ and $\gamma_{\zeta''}(T_1^*)$
belong to the thin horizontal band $\R\times [b-\kappa\e^\alpha,b]$.
Thanks to the property \eqref{eq:gammazeta_close_segment}
we deduce
that the set
\begin{align*}
\Gamma =\gamma_{\zeta'}([T_1^{\zeta'},T_1^*])
\cup\gamma_{\zeta''}([T_1^{\zeta''},T_1^*])\,,
\end{align*}
is contained in the thin horizontal band $\R\times [b-2\kappa\e^\alpha,b]$.
Moreover, the orthogonal projection of $\Gamma$ onto the line $\R\times \lbrace b\rbrace$
contains $[\xi_1(\zeta'),\xi_1(\zeta'')]$ minus an interval of size at most $2\kappa\e^\alpha$,
so that projection must intersect the
interval $[\xi-\kappa\e^\alpha,\xi+\kappa\e^\alpha]\times\lbrace b\rbrace$.
Thus we can find $\tilde\zeta\in \lbrace \zeta',\zeta''\rbrace$ 
and $\tilde T\in [T_1^{\tilde\zeta},T_1^*]$ such that
\begin{align*}
x=\gamma_{\tilde\zeta}(\tilde T) \in B_{3\kappa\e^{\alpha}}((\xi,b))
\subset B_{K\e^\alpha}((\xi,b))
\,,
\end{align*}
provided $K\geq 3\kappa$.
The curve $\gamma(t) =\gamma_{\tilde\zeta}(\tilde T +t)$ satisfies the conclusion of Lemma~\ref{l:across_strip},
with $T=T_2^{\tilde\zeta}-\tilde T$ and $\xi'=\xi_2(\tilde\zeta)$.
\end{proof}

\begin{proof}[Proof of Theorem~\ref{t:bc}]
	If $m\in B^{1/3}_{q,\infty}$ for some $q>\frac{\lt(47+\sqrt{553}\rt)}{12}$, then,
	applying Lemma \ref{L10},
	there exists a finite set $X_{\e}\subset B_1$ such that 
	\begin{align*}
		|m_\e|\geq \frac 12 \text{ in } B_2 \setminus \bigcup_{x\in X_\e}B_{5\e}(x),\qquad\card(X_\e)\lesssim \|m\|_{B^{1/3}_{q,\infty}(B_2)}^q \e^{\frac{q}{3}-2}\,.
	\end{align*}
	Further, we have $\nu\in L^p$ for $p=q/3$ by \cite[Proposition 4.2]{LP23facto}. One can check directly that, for $5.876\approx(47+\sqrt{553})/12<q\leq 6$, we have
	\begin{equation*}
		2-\frac q3 < \alpha=\alpha_p = \frac{p-1}{18p-12}\,.
	\end{equation*}
	Then, for small enough $\e>0$,
	we can find 
	\begin{align*}
		0<\e<a_N<b_N<a_{N-1}<b_{N-1}<\cdots < a_1<b_1\leq 1,
	\end{align*}
	such that 
	\begin{align}
		&|m_\e|\geq \frac 12 \quad\text{in }\bigg(\R\times\bigcup_{j=1}^N (a_j,b_j) \bigg)\cap B_2\,,
		\nn\\
		&
		N\leq\card(X_\e),\quad
		b_j-a_j > 2K \e^\alpha\,,
		\quad
		\mathcal L^1\bigg([0,1]\setminus \bigcup_{j=1}^N (a_j,b_j) \bigg)
		\leq \e^\delta\,,\label{eq:strips}
	\end{align}
	where $\delta=(\alpha-2 +q/3)/2>0$.
	
	Then, inductively applying Lemma~\ref{l:across_strip} on each strip 
	$\lbrace a_j<x_2<b_j\rbrace$ starting from the point 
$(\xi_1, b_1)$ with $\xi_1=0$, we build integral curves $\gamma_j: [0, T_j]\to B_1$ such that
	\begin{align*}
		&\dot\gamma_j =\nabla u_\e(\gamma_j)\,,\quad
		\gamma_j(0)=X_j\,,\quad
		\gamma_j(T_j)=(\xi'_j,a_j)=Y_j\,,
		\\
		&
		X_j\in B_{K\ep^\alpha}((\xi_j, b_j))\,,
		\quad
		|\xi'_j|<\sqrt{1-a_j^2}\,,
		\quad
		\xi_{j+1} = \xi'_j\,,
	\end{align*}
 for $j=1,\ldots,N$. Finally, we set 
 $\zeta = \xi'_N$ and write
	\begin{align}\label{eq:u_ep}
		u_\ep(\zeta, 0)-u_\ep(0,1)&=u_\ep(\zeta,0)-u_\ep(Y_{N})+\sum_{j=1}^N \lt(u_\ep(Y_j)-u_\ep(X_j)\rt)\nn\\
		&\quad+\sum_{j=2}^{N} \lt(u_\ep(X_j)-u_\ep(Y_{j-1})\rt)+(u_\ep(X_1)-u_\ep(0,1)).
	\end{align}

	Using the increasing property \eqref{eq:gammazeta_uepsincrease} of the integral curves $\gamma_j$ and \eqref{eq:strips}, we have
	\begin{equation*}
		\sum_{j=1}^N\lt(u_\ep(Y_j)-u_\ep(X_j)\rt) \geq \sum_{j=1}^N(b_j-a_j)-N(K+\kappa)\ep^\alpha \geq  \sum_{j=1}^N(b_j-a_j) 
		-  \ep^\delta\,,
	\end{equation*}
for small enough $\e>0$.
	Next, since $u_\ep$ is 1-Lipschitz and using the properties \eqref{eq:strips} and $X_j\in B_{K\ep^\alpha}((\xi_j, b_j))$, we deduce that
	\begin{align*}
		&|u_\ep(\zeta,0)-u_\ep(Y_N)|+\sum_{j=2}^{N}|u_\ep(X_j)-u_\ep(Y_{j-1})|+|u_\ep(X_1)-u_\ep(0,1)|\\
		&\quad\quad\leq \mathcal L^1\bigg([0,1]\setminus \bigcup_{j=1}^N (a_j,b_j) \bigg) + N\, K\ep^\alpha \leq \e^\delta\,.
	\end{align*}
	Putting the above two estimates into \eqref{eq:u_ep} and using $|u_\ep(0,1)|\leq \ep$, we obtain
	\begin{equation*}
		u_\ep(\zeta,0)\geq 1-c\e^\delta.
	\end{equation*}
	Letting $\e\to 0$ we deduce that the supremum of $u$ on $B_1$ is at least $1$, which forces $u(x)=1-|x|$ in $B_1$. This translates to \eqref{eq:int_B1_m} through $\na u=im$ as desired.
\end{proof}

%
%
%
%

%
%
%
%
%

\bibliographystyle{acm}
\bibliography{ref_besov6}

\begin{thebibliography}{10}

\bibitem{AG87}
{\sc Aviles, P., and Giga, Y.}
\newblock A mathematical problem related to the physical theory of liquid
  crystal configurations.
\newblock In {\em Miniconference on geometry and partial differential
  equations, 2 ({C}anberra, 1986)}, vol.~12 of {\em Proc. Centre Math. Anal.
  Austral. Nat. Univ.} Austral. Nat. Univ., Canberra, 1987, pp.~1--16.

\bibitem{CC10}
{\sc Caffarelli, L.~A., and Crandall, M.~G.}
\newblock Distance functions and almost global solutions of eikonal equations.
\newblock {\em Commun. Partial Differ. Equations 35}, 3 (2010), 391--414.

\bibitem{CLY24}
{\sc Chuah, C.~Y., Lang, J., and Yao, L.}
\newblock Note about non-compact embeddings between {Besov} spaces.
\newblock {\em arXiv:2410.10731\/} (2024).

\bibitem{dafermos06}
{\sc Dafermos, C.~M.}
\newblock Continuous solutions for balance laws.
\newblock {\em Ric. Mat. 55}, 1 (2006), 79--91.

\bibitem{DLI15}
{\sc De~Lellis, C., and Ignat, R.}
\newblock A regularizing property of the {$2D$}-eikonal equation.
\newblock {\em Comm. Partial Differential Equations 40}, 8 (2015), 1543--1557.

\bibitem{DLO03}
{\sc De~Lellis, C., and Otto, F.}
\newblock Structure of entropy solutions to the eikonal equation.
\newblock {\em J. Eur. Math. Soc. (JEMS) 5}, 2 (2003), 107--145.

\bibitem{DKMO01}
{\sc DeSimone, A., M\"{u}ller, S., Kohn, R.~V., and Otto, F.}
\newblock A compactness result in the gradient theory of phase transitions.
\newblock {\em Proc. Roy. Soc. Edinburgh Sect. A 131}, 4 (2001), 833--844.

\bibitem{GL20}
{\sc Ghiraldin, F., and Lamy, X.}
\newblock Optimal {Besov} differentiability for entropy solutions of the
  eikonal equation.
\newblock {\em Commun. Pure Appl. Math. 73}, 2 (2020), 317--349.

\bibitem{ignat12survey}
{\sc Ignat, R.}
\newblock Singularities of divergence-free vector fields with values into
  {$S^1$} or {$S^2$}. {A}pplications to micromagnetics.
\newblock {\em Confluentes Math. 4}, 3 (2012), 1230001, 80.

\bibitem{ignat12}
{\sc Ignat, R.}
\newblock Two-dimensional unit-length vector fields of vanishing divergence.
\newblock {\em J. Funct. Anal. 262}, 8 (2012), 3465--3494.

\bibitem{ignat24short}
{\sc Ignat, R.}
\newblock {A short proof of the {$\mathcal C^{1,1}$} regularity for the eikonal
  equation}.
\newblock {\em arXiv:2409.05204\/} (2024).

\bibitem{JOP02}
{\sc Jabin, P.-E., Otto, F., and Perthame, B.}
\newblock Line-energy {Ginzburg}-{Landau} models: zero-energy states.
\newblock {\em Ann. Sc. Norm. Super. Pisa, Cl. Sci. (5) 1}, 1 (2002), 187--202.

\bibitem{JP01}
{\sc Jabin, P.-E., and Perthame, B.}
\newblock Compactness in {Ginzburg}-{Landau} energy by kinetic averaging.
\newblock {\em Commun. Pure Appl. Math. 54}, 9 (2001), 1096--1109.

\bibitem{JK00}
{\sc Jin, W., and Kohn, R.~V.}
\newblock Singular perturbation and the energy of folds.
\newblock {\em J. Nonlinear Sci. 10}, 3 (2000), 355--390.

\bibitem{LP23facto}
{\sc Lorent, A., and Peng, G.}
\newblock Factorization for entropy production of the eikonal equation and
  regularity.
\newblock {\em Indiana Univ. Math. J. 72}, 3 (2023), 1055--1105.

\bibitem{triebel83}
{\sc Triebel, H.}
\newblock {\em Theory of function spaces}, vol.~78 of {\em Monogr. Math.,
  Basel}.
\newblock Birkh{\"a}user, Cham, 1983.

\bibitem{vasseur01}
{\sc Vasseur, A.}
\newblock Strong traces for solutions of multidimensional scalar conservation
  laws.
\newblock {\em Arch. Ration. Mech. Anal. 160}, 3 (2001), 181--193.

\end{thebibliography}

\end{document}